\documentclass{amsart}
\usepackage[dvipdfmx]{graphicx}
\numberwithin{equation}{section}
\theoremstyle{definition}
\newtheorem{thm}{Theorem}[section]
\newtheorem{prop}[thm]{Proposition}

\newtheorem{lem}[thm]{Lemma}
\newtheorem{rem}[thm]{Remark}
\newtheorem{cor}[thm]{Corollary}
\newtheorem{ex}[thm]{Example}
\newtheorem*{ack}{Acknowledgments}
\newtheorem*{mt}{Main Theorem}

\def\rk{\mathop{\mathrm{rk}}\nolimits}

\title[Log Enriques surfaces of index 7]{Log Enriques surfaces of index 7 and type $A_{15}$}
\author[S.~Taki]{Shingo Taki}
\address{Department of Mathematics, Tokai University,
4-1-1, Kitakaname, Hiratsuka, Kanagawa, 259-1292, JAPAN}
\email{taki@tsc.u-tokai.ac.jp}
\urladdr{http://www.sm.u-tokai.ac.jp/~taki/}
\date{\today}
\subjclass[2010]{Primary 14J26, 14J28; Secondary 14J50}
\keywords{log Enriques surface, $K3$ surface, non-symplectic automorphism}
\dedicatory{Dedicated to Professor Shigeyuki Kondo on the occasion of his sixtieth birthday}
\thanks{}
\begin{document}

\begin{abstract}
We show that there is only one log Enriques surface of index 7 and type $A_{15}$.
\end{abstract}

\maketitle

\tableofcontents

\section{Introduction}\label{Introduction}
We will work over $\mathbb{C}$, the field of complex numbers, throughout this paper.
Let $Z$ be a normal algebraic surface with at worst log terminal singularities.
$Z$ is called \textit{log Enriques} if the irregularity $\dim H^{1}(Z, \mathcal{O}_{Z})=0$
and a positive multiple $IK_{Z}$ of a canonical Weil divisor $K_{Z}$ is linearly equivalent to zero.
The smallest integer $I>0$ satisfying $IK_{Z} \sim 0$ is called the \textit{index} of $Z$.
Without loss of generality, we assume that a log Enriques surface has no Du Val singular points,
because if $Z' \to Z$ is the minimal resolution of all Du Val singular points of $Z$ then
$Z'$ is also a log Enriques surface of the same index of $Z$.

Let $Z$ be a log Enriques surface of index $I$.
The Galois $\mathbb{Z}/I\mathbb{Z}$-cover
\[ \pi: Y:= \text{Spec}_{{\mathcal{O}}_{Z}} \left( \bigoplus_{i=0}^{I-1}\mathcal{O}_{Z}(-iK_{Z}) \right) \to Z \]
is called the (global) \textit{canonical covering}.
Note that $Y$ is either an abelian surface or a $K3$ surface with at worst Du Val singular points, 
and that $\pi$ is unramified over $Z\setminus \text{Sing}(Z)$. If $Y$ is an abelian surface then $I=3$ or 5.
See also \cite{Z1} for details. A log Enriques surface $Z$ is \textit{of type} $A_{m}$ or $D_{n}$ if, by definition,
its canonical cover $Y$ has a  singular point of type $A_{m}$ or $D_{n}$, respectively.

It is interesting to consider the index $I$ of a log Enriques surface.
Blache \cite{Bl} proved that $I\leq 21$. Thus if $I$ is prime then 
$I=2, 3, 5, 7, 11, 13, 17$ or 19.

\begin{thm}[\cite{logEnriques19, logEnriques18, OZorder5, 11, 13-19}]\label{rekishi}
The followings hold:
\begin{itemize}
\item[(1)] There is one log Enriques surface of type $D_{19}$ (resp. $A_{19}$, $D_{18}$), up to isomorphism.
\item[(2)] There are  two log Enriques surfaces of type $A_{18}$, up to isomorphism.
\item[(3)] There are  two log Enriques surfaces of index 5 and type $A_{17}$, up to isomorphism.
\end{itemize}
The followings do not refer to singular points.
But these determine log Enriques surfaces with large prime indices:
\begin{itemize}
\item[(4)]  There are  two maximal log Enriques surfaces of index 11, up to isomorphism.
\item[(5)] If $I$=13, 17 or 19 then there is a unique log Enriques surface of index $I$, up to isomorphism.
\end{itemize}
\end{thm}
\begin{rem}
If a log Enriques surface is of type $A_{19}$ (resp. $A_{18}$, $D_{18}$ or $D_{19}$) 
then its index is 2 (resp. 3).
\end{rem}

Let $\omega _{X}$ be a nowhere vanishing holomorphic $2$-form on an algebraic $K3$ surface $X$
and $\sigma$ an automorphism on $X$ of finite order $I$.
It is called \textit{non-symplectic} if and only if it satisfies 
$\sigma^{\ast }\omega _{X} =\zeta _{I} \omega _{X}$ where $\zeta _{I}$ is a primitive $I$-th root of unity.
To prove Theorem \ref{rekishi}, we studied non-symplectic automorphisms of $K3$ surfaces, 
because the canonical covering $\pi$ is a cyclic Galois covering of order $I$ which 
acts faithfully on the space $H^{0}(Y,\mathcal{O}_{Y}(K_{Y}))$.
And we have gotten the following.

\begin{thm}[{\cite{logEnriques19, OZorder5, 11, 13-19}}]\label{32511}
Let $\sigma_{I}$ be a non-symplectic automorphism of order $I$ on a $K3$ surface $X_{I}$ and
$X_{I}^{\sigma_{I}}$ be the fixed locus of $\sigma_{I}$;
$X_{I}^{\sigma_{I}}=\{x\in X_{I}|\sigma_{I}(x)=x \}$.
Then the followings hold:
\begin{itemize}
\item[(1)]
If $X_{3}^{\sigma_{3}}$ 
consists of only (smooth) rational curves and possibly some isolated points, 
and contains at least 6 rational curves then 
a pair ($X_{3}$, $\langle \sigma_{3} \rangle$) is unique up to isomorphism.
\item[(2)]
If $X_{2}^{\sigma_{2}}$ 
consists of only (smooth) rational curves 
and contains at least 10 rational curves then 
a pair ($X_{2}$, $\langle \sigma_{2} \rangle$) is unique up to isomorphism.
\item[(3)]
If $X_{5}^{\sigma_{5}}$ 
contains no curves of genus $\geq 2$, but contains at least 3 rational curves
then a pair ($X_{5}$, $\langle \sigma_{5} \rangle$) is unique up to isomorphism.
\item[(4)]
Put $M:=\{x\in H^{2}(X_{11},\mathbb{Z})| \sigma_{11}^{\ast} (x) =x\}$.
A pair ($X_{11}$, $\langle \sigma_{11} \rangle$) is unique up to isomorphism
if and only if $M=U\oplus A_{10}$.
\item[(5)] Pairs ($X_{13}$, $\langle \sigma_{13} \rangle$),  ($X_{17}$, $\langle \sigma_{17} \rangle$) and
($X_{19}$, $\langle \sigma_{19} \rangle$) are unique up to isomorphism, respectively.
\end{itemize}
\end{thm}

These theorems miss the case of $I=7$.
The main purpose of this paper is to prove the following theorem:
\begin{mt}
\begin{itemize}
\item[(1)] There is, up to isomorphism, only one log Enriques surface of index 7 and type $A_{15}$.
\item[(2)] If $X_{7}^{\sigma_{7}}$ consists of only smooth rational curves and some isolated points
and contains at least 2 rational curves then
a pair ($X_{7}$, $\langle \sigma_{7} \rangle$) is unique up to isomorphism.
\end{itemize}
\end{mt}

We summarize the contents of this paper.
In Section \ref{order7}, we study $K3$ surfaces with a non-symplectic automorphism and prove Main Theorem (2).
In Section \ref{correspondence}, we  see uniqueness of the $K3$ surface with a non-symplectic automorphism of order 7 
which is constructed from log Enriques surfaces of index 7 and type $A_{15}$.
And we give an example of a construction for such a log Enriques surface 
from a $K3$ surface with a non-symplectic automorphism 7.
In Section \ref{sublattice}, we study a sublattice of type $A_{15}$ in the N\'{e}ron-Severi lattice of a $K3$ surface
and give a proof of Main Theorem (1).

\begin{ack}
The author was partially supported by 
Grant-in-Aid for Young Scientists (B) 15K17520 from JSPS.
\end{ack}

\section{$K3$ surfaces with a non-symplectic automorphism of order 7}\label{order7}
In this section, we collect some basic results 
for non-symplectic automorphisms on a $K3$ surface. 
For the details, see \cite{Ni2} and \cite{AST}, and so on.

For a $K3$ surface $X$, we denote by $S_{X}$ and $T_{X}$
the N\'{e}ron-Severi lattice and the transcendental lattice, respectively.

\begin{lem}\label{sayou}
Let $\sigma$ be a non-symplectic automorphism of order $I$ on $X$.
Then 
\begin{itemize}
\item[(1)] The eigen values of $\sigma^{\ast }\mid T_{X}$ are the 
primitive $I$-th roots of unity, hence
$\sigma^{\ast }\mid T_{X}\otimes \mathbb{C}$ can be diagonalized as:
\[ \begin{pmatrix} 
\zeta_{I} E_{q} & 0 & \cdots & \cdots & \cdots & 0 \\ 
\vdots &  & \ddots &  &  & \vdots \\ 
\vdots &  &  & \zeta_{I}^{n} E_{q} &  & \vdots \\ 
\vdots &  &  &  & \ddots  & 0 \\ 
0 & \cdots & \cdots & \cdots & 0 & \zeta_{I}^{I-1} E_{q} \\ 
\end{pmatrix}, \]
where $E_{q}$ is the identity matrix of size $q$ 
and $1\leq n\leq I-1$ is co-prime with $I$.

\item[(2)] Let $P^{i,j}$ be an isolated fixed point of $\sigma$ on $X$. 
Then $\sigma^{\ast }$ can be written as 
\[ \begin{pmatrix}  \zeta_{I} ^{i} & 0 \\ 0 & \zeta_{I} ^{j}  \end{pmatrix}  \hspace{5mm} (i+j \equiv 1 \mod I) \]
under some appropriate local coordinates around $P^{i,j}$.
\item[(3)] Let $C$ be an irreducible curve in $X^{\sigma}$ and $Q$ a point on $C$. 
Then $\sigma^{\ast }$ can be written as
\[ \begin{pmatrix}  1 & 0 \\ 0 & \zeta_{I}   \end{pmatrix} \] 
under some appropriate local coordinates around $Q$. 
In particular, fixed curves are non-singular.
\end{itemize}
\end{lem}

Lemma \ref{sayou} (1) implies that $\Phi (I)$ divides $\rk T_{X}$, where $\Phi$ is the Euler function.
Lemma \ref{sayou} (2) and (3) imply that 
the fixed locus of $\sigma$ is either empty or the disjoint union of non-singular curves and isolated points:
\[ X^{\sigma}=\{ P_{1}^{i_{1}, j_{1}}, \dots , P_{M}^{i_{M}, j_{M}} \} \amalg C_{1} \amalg \dots \amalg C_{N}, \]
where $P_{k}^{i_{k},j_{k}}$ is an isolated fixed point and $C_{l}$ is a non-singular curve.

The global Torelli Theorem gives the following.
\begin{rem}[{\cite[Lemma (1.6)]{machida-oguiso}}]\label{Torelli}
Let $X$ be a $K3$ surface and $g_{i}$ ($i=1$, $2$) automorphisms of $X$
such that $g_{1}^{\ast}|S_{X}=g_{2}^{\ast}|S_{X}$ and that
$g_{1}^{\ast}\omega _{X}=g_{2}^{\ast}\omega _{X}$.
Then $g_{1}=g_{2}$ in Aut ($X$).
\end{rem}

The Remark says that for study of non-symplectic automorphisms, 
the action on $S_{X}$ is important. Hence the invariant lattice 
$S_{X}^{\sigma}:=\{x\in S_{X}| \sigma^{\ast} (x) =x\}$ plays an essential role
for the classification of non-symplectic automorphisms.

In the following, we denote $\sigma$  a non-symplectic automorphism of order 7
on a $K3$ suface $X$.
The following propositions are keys for Main Theorem (2).

\begin{prop}[\cite{13-19}]\label{vo-non-unimodular}
Assume that $\sigma$ acts trivially on $S_{X}$, hence $S_{X}=S_{X}^{\sigma}$.
If $\Phi (7)=6=\rk T_{X}$ then such a  $K3$ surface is unique.
\end{prop}

\begin{prop}[{\cite[Theorem 6.3]{AST}}]\label{class-ord7}
Then the fixed locus $X^{\sigma}$ is of the form
\begin{equation*}
X^{\sigma}=
\begin{cases}
\{ P_{1}, P_{2}, P_{3} \} \amalg E&  \text{if $S_{X}^{\sigma}=U\oplus K_{7}$,} \\
\{ P_{1}, P_{2}, P_{3} \} &  \text{if $S_{X}^{\sigma}=U(7)\oplus K_{7}$,} \\
\{ P_{1}, P_{2}, \dots , P_{8} \}\amalg E \amalg \mathbb{P}^{1} & \text{if $S_{X}^{\sigma}=U\oplus E_{8}$,}\\
\{ P_{1}, P_{2}, \dots , P_{8} \}\amalg  \mathbb{P}^{1} & \text{if $S_{X}^{\sigma}=U(7)\oplus E_{8}$,}\\
\{ P_{1}, P_{2}, \dots , P_{13} \}\amalg \mathbb{P}^{1}\amalg \mathbb{P}^{1} & \text{if $S_{X}^{\sigma}=U\oplus E_{8}\oplus A_{6}$.}
\end{cases}
\end{equation*}
Here $E$ is a non-singular curve of genus 1, 
$A_{6}$ or $E_{8}$ are the negative-definite root lattice of type $A_{6}$ or $E_{8}$ respectively. 
We denote by $U$ the even indefinite unimodular lattice of rank 2 and 
$U(7)$ the lattice whose bilinear form is the one on $U$ multiplied by 7. 
The even negative-definite lattice $K_{7}$ is given by Gram matrix 
$\begin{pmatrix}
-4 & 1 \\
1 & -2
\end{pmatrix}$.
Moreover the number of isolated fixed points of type 
$P^{2,6}$ (resp. $P^{3,5}$ or $P^{4,4}$) is 
$(\rk S_{X}^{\sigma}+2)/3$ (resp. $(\rk S_{X}^{\sigma}-1)/3$ or $(\rk S_{X}^{\sigma}-4)/6$.)
\end{prop}

In the following, we treat a pair ($X$, $\langle \sigma \rangle$)
whose the fixed locus $X^{\sigma}$ 
consists of smooth rational curves and isolated points, 
and contains at least 2 rational curves. 
We show that the pair ($X$, $\langle \sigma \rangle$)
is unique up to isomorphism.

\begin{prop}\label{id-7-S}
The automorphism $\sigma$ acts trivially on $S_{X}$.
\end{prop}
\begin{proof}
Since $X^{\sigma}$ has at least 2 rational curves,
$X^{\sigma}=\{P_{1}, P_{2}, \dots ,P_{13} \} \amalg \mathbb{P}^{1} \amalg \mathbb{P}^{1}$
and $S_{X}^{\sigma}=U\oplus E_{8}\oplus A_{6}$ by 
Proposition \ref{class-ord7}.
We know that $\rk T_{X}\geq 6$ by Lemma \ref{sayou} (1) and 
$\rk S_{X}\geq 16$ since it contains the invariant lattice
$S_{X}^{\sigma}$ which is of rank 16. 
This gives $\rk T_{X}\leq 6$ so that $\rk T_{X}=6$ and $\rk S_{X}= 16$, 
hence $S_{X}$ coincides with $S_{X}^{\sigma}$. 
This implies that the action of $\sigma$ is trivial on the $S_{X}$.
\end{proof}

The following Corollary follows from Proposition \ref{id-7-S} and Proposition \ref{class-ord7}.

\begin{cor}\label{corbunrui}
Under the above hypothesis, 
$S_{X}=U\oplus E_{8}\oplus A_{6}$, $T_{X}=U \oplus U \oplus K_{7}$ 
and the fixed locus $\sigma$ has 2 non-singular rational curves and 13 isolated points: 
$X^{\sigma}=\{P_{1}, P_{2}, \dots ,P_{13} \} \amalg \mathbb{P}^{1} \amalg \mathbb{P}^{1}$.
\end{cor}

We recall that the dimension of a moduli space of $K3$ surfaces with 
a non-symplectic automorphism of order 7 is $\rk T_{X} / \Phi(7)-1$
(see also \cite[Section 11]{DK}).
In our case, its dimension is 0. Indeed we have the following.

\begin{thm}\label{ord7}
A pair ($X$, $\langle \sigma \rangle$) is unique up to isomorphism, hence
Main Theorem (2) holds.
\end{thm}
\begin{proof}
It follows from 
Proposition \ref{id-7-S}, Proposition \ref{vo-non-unimodular}.and Remark \ref{Torelli}.
\end{proof}

\begin{ex}[{\cite[Example 6.1 (3)]{AST}}]\label{exk3w7}
Put 
\[ X_{\text{AST}}:y^{2}=x^{3}+\sqrt[3]{-27/4}x+t^{7}-1, \ 
\sigma_{\text{AST}} (x,y,t) = (x, y, \zeta _{7}t). \]
Then $X_{\text{AST}}$ is a $K3$ surface with $S_{X_{\text{AST}}}=U\oplus E_{8}\oplus A_{6}$
and $\sigma_{\text{AST}}$ is a non-symplectic automorphism of order 7.
Note that $X_{\text{AST}}$ has one singular fiber of type I$_{7}$ over $t=0$, one singular fiber of type II$^{\ast}$ over $t=\infty$
and 7 singular fibers of type I$_{1}$ over $t^{7}=1$.
\end{ex}

\begin{ex}[{\cite[(7.5)]{Kondo1}\label{exk3w7-2}}]
Put
\[X_{\text{Ko}}:y^{2}=x^{3}+t^{3}x+t^{8}, \ \sigma_{\text{Ko}}(x,y,t)=(\zeta _{7}^{3}x,\zeta _{7}y,\zeta _{7}^{2}t) \]
Then $X_{\text{Ko}}$ is a $K3$ surface with $S_{X_{\text{Ko}}}=U\oplus E_{8}\oplus A_{6}$
and $\sigma_{\text{Ko}}$ is a non-symplectic automorphism of order 7.
Note that $X_{\text{Ko}}$ has one singular fiber of type III$^{\ast}$ over $t=0$, one singular fiber of type IV$^{\ast}$ over $t=\infty$
and 7 singular fibers of type I$_{1}$ over $4+27t^{7}=0$.
Moreover the rank of the Mordell-Weil group is 1.
\end{ex}

\begin{rem}\label{rigidity}
The local actions of a non-symplectic automorphism of order 7
at the intersection points of the rational curves appear 
in the following order:
\[ \dots, \begin{pmatrix}  1 & 0 \\ 0 & \zeta_{7}   \end{pmatrix}, \begin{pmatrix}  \zeta_{7} ^{6} & 0 \\ 0 & \zeta_{7} ^{2}  \end{pmatrix},
\begin{pmatrix}  \zeta_{7} ^{5} & 0 \\ 0 & \zeta_{7} ^{3}  \end{pmatrix}, \begin{pmatrix}  \zeta_{7} ^{4} & 0 \\ 0 & \zeta_{7} ^{4}  \end{pmatrix},\]
\[\begin{pmatrix}  \zeta_{7} ^{3} & 0 \\ 0 & \zeta_{7} ^{5}  \end{pmatrix}, \begin{pmatrix}  \zeta_{7} ^{2} & 0 \\ 0 & \zeta_{7} ^{6}  \end{pmatrix},
\begin{pmatrix}  \zeta_{7} & 0 \\ 0 & 1  \end{pmatrix}, \begin{pmatrix}  1 & 0 \\ 0 & \zeta_{7}   \end{pmatrix}, \dots. \]
\end{rem}

\section{A correspondence between log Enriques surfaces and $K3$ surfaces}\label{correspondence}
Let $Z$ be a log Enriques surface of index 7 and  type $A_{15}$ without Du Val singularities, 
$\pi:Y \to Z$ the canonical covering of $Z$ and $f:X\to Y$ the minimal resolution.
Note that  $X$ is uniquely determined up to isomorphism. Recall that $X$ is a $K3$ surface (see also \cite[Theorem 4.1]{Z1} ) 
and $\sigma$ is a non-symplectic automorphism of order 7 induced by $\pi$.

\begin{lem}\label{del-sta}
Let $\Delta$ be the exceptional divisor of the minimal resolution $f$.
Then every component of $\Delta$ is $\sigma$-stable.
\end{lem}
\begin{proof}
Note that $\Delta$ is $\sigma$-stable and a liner chain of Dynkin type $A_{15}$.
It follows from the fact that the order of symmetry of $\Delta$ is co-prime with 7.
\end{proof}

\begin{prop}\label{le-k3uni}
The pair ($X$, $\langle \sigma \rangle$) is unique up to isomorphism.
\end{prop}
\begin{proof}
Since $\pi$ is unramified over $Z\setminus \text{Sing}(Z)$, every fixed curve by $\sigma$ in $X$ is contained in $\Delta$.
Hence $X^{\sigma}$ contains only smooth rational curves and isolated fixed points.
On the other hand, each component of $\Delta$ has two isolated fixed points or is pointwisely fixed by $\sigma$ by Lemma \ref{del-sta}.

We remark that $\Delta$ consists of 15 smooth rational curves.
If $X^{\sigma}$ has one or less smooth rational curves then
$X^{\sigma}$ has at least 14 or more isolated fixed points.
This is a contradiction by Proposition \ref{class-ord7}.
Then $X^{\sigma}$ consist of exactly two smooth rational curves, and
the claim follows from Theorem \ref{ord7}.
\end{proof}

By the Proposition, if $X^{\sigma}$ consists of only smooth rational curves 
and some isolated points and contains at least 2 rational curves then
a pair ($X$, $\langle \sigma \rangle$) corresponds to log Enriques surfaces of index 7 and type $A_{15}$.
Hence we may identify it with the pair in Example \ref{exk3w7} or Example \ref{exk3w7-2}.
We construct a log Enriques surface of index 7 and type $A_{15}$ 
from a $K3$ surface with a non-symplectic automorphism 7 given by Example \ref{exk3w7}.

\begin{ex}\label{triple}
We consider the pair  ($X_{\text{AST}}$, $\langle \sigma_{\text{AST}} \rangle$) in Example \ref{exk3w7}.
Let $f:X_{\text{AST}}\to Y$ be the contraction of the following rational tree $\Delta _{\text{AST}}$ of Dynkin type $A_{15}$ to a point $Q$:
\[ \Gamma_{2}-\Gamma_{3}-\Gamma_{4}-\Gamma_{5}-\Gamma_{6}-\Gamma_{7}-S
-\Theta_{1}-\Theta_{2}-\Theta_{3}-\Theta_{4}-\Theta_{5}-\Theta_{6}-\Theta_{7}-\Theta_{8}, \]
where $S$ is a cross-section, $\Gamma_{i}$ is a component of a singular fiber of type I$_{7}$
and $\Theta_{j}$ is a component of a singular fiber of type II$^{\ast}$.
Here a singular fiber of type  I$_{7}$ is given by $\sum_{i=1}^{7}\Gamma_{i}$
which $\Gamma_{7}$ meets $S$, 
and a singular fiber of type II$^{\ast}$ is given by $\sum_{j=1}^{6}j\Theta_{j}+4\Theta_{7}+2\Theta_{8}+3\Theta_{9}$.
Hence $\Gamma_{7}$ and $\Theta_{6}$ are fixed curves of $\sigma_{\text{AST}}$.

\begin{figure}[h]
\begin{center}
\includegraphics[width=10.5cm]{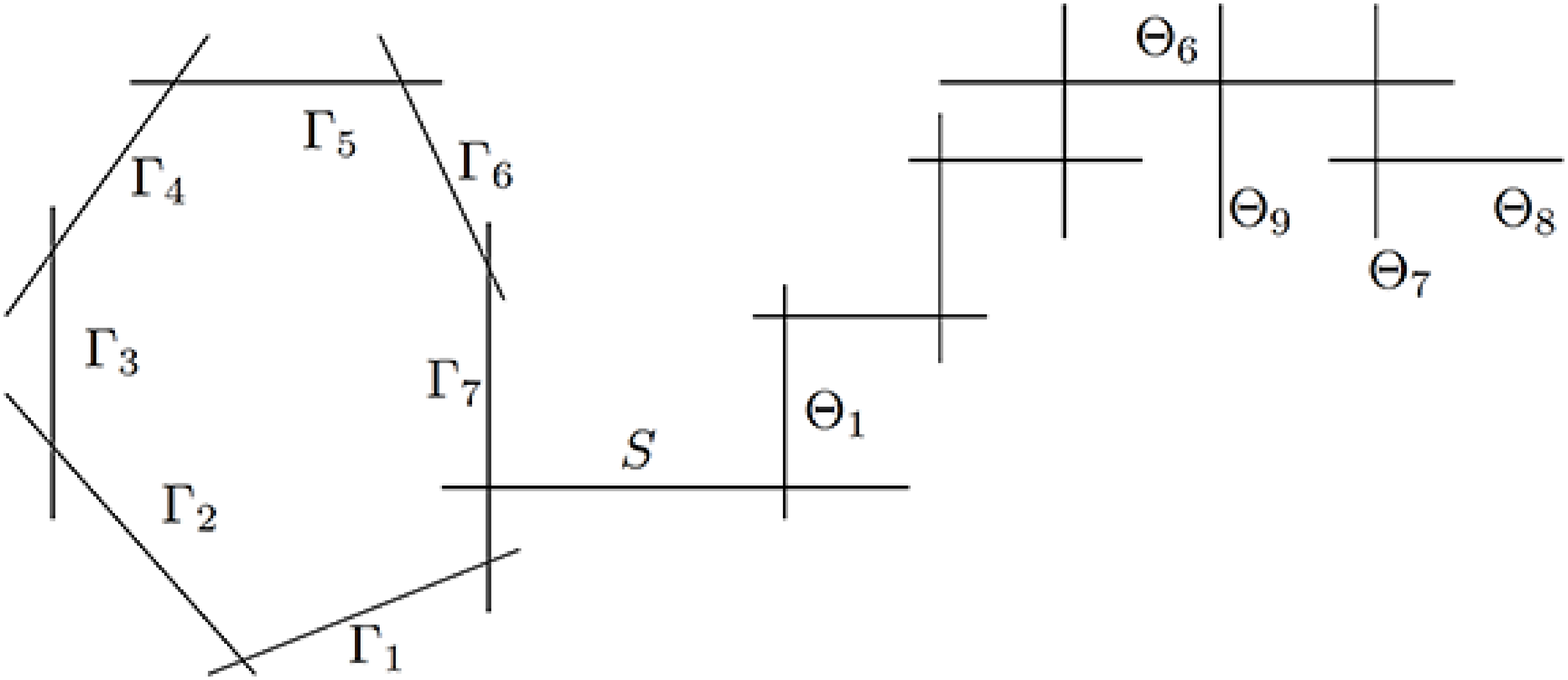}


\end{center}
\caption{Singular fibers of Example \ref{exk3w7}}
\end{figure}

Then $\sigma_{\text{AST}}$ induces an automorphism $\tau$ on $Y$ so that $Y^{\tau}:=\{Q, f(P) \}$
where $P$ is the isolated fixed point of type $P^{2,6}$ on $\Theta_{9}$.
Now the quotient surface $Z_{\text{AST}}:=Y/\tau$ is a log Enriques surface of index 7 and type $A_{15}$.
Note that $Z_{\text{AST}}$ has exactly two singular points
under the two fixed points $Q$ and $f(P)$.
\end{ex}

\begin{rem}
We can fined 13 isolated fixed points and 2 fixed curves of  $\sigma_{\text{AST}}$
on these singular fibers by Remark \ref{rigidity}.
Note that $\Theta_{6}$ is  pointwisely fixed by $\sigma_{\text{AST}}$.

6 isolated fixed points of type $P^{2,6}$ are intersection points of $\Gamma_{1}$ and $\Gamma_{2}$, 
$\Gamma_{5}$ and $\Gamma_{6}$,  $S$ and $\Theta_{1}$, $\Theta_{4}$ and $\Theta_{5}$, 
$\Theta_{7}$ and $\Theta_{8}$, and a point on $\Theta_{9}$.
5 isolated fixed points of type $P^{3,5}$ are intersection points of $\Gamma_{2}$ and $\Gamma_{3}$, 
$\Gamma_{4}$ and $\Gamma_{5}$, $\Theta_{1}$ and $\Theta_{2}$, $\Theta_{3}$ and $\Theta_{4}$,
and a point on $\Theta_{8}$.
2 isolated fixed points of type $P^{4,4}$ are intersection points of $\Gamma_{3}$ and $\Gamma_{4}$, 
and $\Theta_{2}$ and $\Theta_{3}$.
Then $\Gamma_{7}$ is a fixed curve.
\end{rem}

\begin{rem}
In \cite[Example 6.13]{Z1}, we constructed a log Enriques surface of index 7 and type $A_{15}$
that do not use $K3$ surfaces.
Of course we can also see it by contracting some divisors on the minimal resolution of the quotient surface $X/ \sigma$.
\end{rem}

\section{Sublattices of type $A_{15}$}\label{sublattice}

Assume that a pair ($X$, $\langle \sigma \rangle$) corresponds to a log Enriques surface of index 7 and type $A_{15}$.
In the following we write $\Delta=\sum_{i=1}^{15}C_{i}$, which is of Dykin type $A_{15}$
and employ the same symbol $\Delta$ for the sublattice of $S_{X}$
generated by the irreducible component of $\Delta$.

Since the rank of the orthogonal complement $\Delta^{\perp}$ of $\Delta$ in $S_{X}$ is 1,
we may write $\Delta^{\perp}=\mathbb{Z}H$ and assume $H$ is a nef and big divisor.
In fact if $X\to Y$ is the contraction of $\Delta$ then the Picard number of $Y$ is 1, 
thus we can take $H$ as the pull back of the ample generator of $S_{Y}$.

\begin{lem}\label{not-primitive}
The lattice $\Delta$ is a primitive sublattice of $S_{X}$.
\end{lem}

\begin{proof}
Assume that the lattice $\Delta$ is not a primitive sublattice of $S_{X}$.
Let $\overline{\Delta}$ be the primitive closure of $\Delta$  in $S_{X}$.
Since $16=|\det (\Delta)|=[\overline{\Delta}:\Delta]^{2}\det(\overline{\Delta})$,
we have 
$|\det(\overline{\Delta})|=4$ and $[\overline{\Delta}:\Delta]=2$, or 
$|\det(\overline{\Delta})|=1$ and $[\overline{\Delta}:\Delta]=4$.

If $|\det(\overline{\Delta})|=1$, that is, $\overline{\Delta}$ is unimodular then we have 
$|H^{2}|=|\det(S_{X})|=7$ because of Corollary \ref{corbunrui} and $S_{X}=\overline{\Delta}\oplus \mathbb{Z}H$.
This contradicts for the fact that $S_{X}$ is an even lattice.
Then $|\det(\overline{\Delta})|=4$ and $[\overline{\Delta}:\Delta]=2$.

Thus we can find a non-empty subset $J$ of $\{1, 2, \dots , 15\}$ such that 
$\frac{1}{2}\sum_{j\in J}C_{j}$ is contained in $S_{X}$.
Since $C_{i}$ is a non-singular rational curve, $\sharp J=8$ or 16 by \cite[Lemma 3]{Ni} or \cite[Lemma 3.3]{Morrison}.
But there exists an element $i \in \{1, 2, \dots , 15\}$ such that the intersection number 
$\left ( \frac{1}{2}\sum_{j\in J}C_{j} \right )C_{i}$ is not an integer.
This is a contradiction.
\end{proof}

\begin{lem}\label{primitive}
There exists an element $h$ in $S_{X}$ such that $S_{X}=\Delta+\mathbb{Z}h$
and the followings hold:
\begin{itemize}
\item[(1)] The index $[S_{X}:\Delta\oplus \mathbb{Z}H]=16$ and $H^{2}=112$. 
\item[(2)] There exist integers $a_{i}$ such that $h=H/16+\sum_{i=1}^{15}(a_{i}/16)C_{i}$.
\item[(3)] Put $h_{+}:=(H+3\sum_{i=1}^{15}iC_{i})/16$. Then $h\equiv h_{+}$ (mod $\Delta$)
and $\{ C_{1}, C_{2}, \dots, C_{15}, h_{+} \}$ is a $\mathbb{Z}$-basis of $S_{X}$.
\end{itemize}
\end{lem}
\begin{proof}
(1) (2) 
Set $n:=[S_{X}:\Delta\oplus \mathbb{Z}H]$. 
Since Corollary \ref{corbunrui} and $|\det(\Delta\oplus \mathbb{Z}H)|=n^{2}|\det(S_{X})|$,
it satisfies $16H^{2}=7n^{2}$.
After replacing $h$ by $-h$ if necessary, we can find integers $a_{i}$ such that $H=nh-\sum_{i=1}^{15}a_{i}C_{i}$.
Note that $(a_{1}/n, \dots , a_{15}/n)$ is the unique solution of the liner system:
\[ \left( h-\sum_{i=1}^{15}x_{i}C_{i}\right)C_{j}=0 \ \ \ (j=1, \dots, 15).  \]
Since the determinant of the Gramm matrix of $\Delta$, that is,  
$\det(C_{i}, C_{j})=-16$, the numbers $16a_{i}/n$ are integers. 
Hence $16H/n=16h-\sum_{i=1}^{15}(16a_{i}/n)C_{i}=rH$ for some integer $r$, 
so $n$ divides 16 ($=|\det (\Delta)|$).
Since $S_{X}$ is an even lattice, $n=8$ ($H^{2}$=28) or $n=16$ ($H^{2}=112$).

Note that $\sum_{i=1}^{15}(a_{i}C_{i})C_{j}=n(-h.C_{j})\equiv 0$ (mod $n$) for all $j$, hence
\[ -2a_{1}+a_{2}\equiv 0,\ \  a_{i-1}-2a_{i}+a_{i+1}\equiv 0 \ (i=2,3, \dots , 14), \ \ a_{14}-a_{15}\equiv 0 \]
(mod $n$). Thus $a_{i}\equiv ia_{1}$ for all $i=1, 2, \dots , 15$ and
 \begin{align*}
\left(h-\frac{1}{n} \sum_{i=1}^{15}(ia_{1}+a_{i})C_{i} \right)^{2}
&= \frac{1}{n^{2}}\left(H-a_{1}\sum_{i=1}^{15}iC_{i} \right)^{2} \\
&= \frac{1}{n^{2}}(H^{2}-16\times 15a_{1}^{2})\\
&= \frac{7}{16}-\frac{16\times 15a_{1}^{2}}{n^{2}}
\end{align*}
is an integer.
This implies that $n=16$ (and $a_{1}\equiv \pm 3$ mod 16).

(3) It follows from the definition of $h_{+}$.
\end{proof}

In order to prove Main Theorem (1), it suffices to show that $Z$ is isomorphic to
the log Enriques surface $Z_{\text{AST}}$ in Example \ref{triple}.
Hence we show that there exist an automorphism $\varphi: X_{\text{AST}}\to X_{\text{AST}}$
such that $\varphi (\Delta) = \Delta _{\text{AST}}$ and $\varphi \circ\sigma_{\text{AST}} =\sigma_{\text{AST}} \circ \varphi $.

\begin{lem}\label{isometry}
Write $\Delta _{\text{AST}}=\sum_{i=1}^{15}D_{i}$ the same way as $\Delta$. 
Put $\Delta_{\text{AST}}^{\perp}=\mathbb{Z}H_{\text{AST}}$ and $\Delta^{\perp}=\mathbb{Z}H$ in $S_{X_{\text{AST}}}$.
Then there exist an isometry $\Phi$ of the lattice $S_{X_{\text{AST}}}$ 
such that $\Phi (\Delta) = \Delta _{\text{AST}}$,  $\Phi (H) = H_{\text{AST}}$
and $\Phi$ preserves the ample cone.
\end{lem}
\begin{proof}
By Lemma \ref{primitive}, $h_{+}$ is uniquely and precisely expressed $H$ and $C_{j}$.
Two natural isometries $\Phi_{1}: \Delta \to \Delta _{\text{AST}}$ and 
$\Phi_{2}: H \to H _{\text{AST}}$ can be extended to an isometry $\Phi : S_{X_{\text{AST}}}\to S_{X_{\text{AST}}}$
such that $\Phi (\Delta) = \Delta _{\text{AST}}$ and $\Phi (H) = H_{\text{AST}}$.

We note that $H-\sum_{i=1}^{15}\alpha_{i}C_{i}$ and its image by $\Phi$, i.e.,
$H_{\text{AST}}-\sum_{i=1}^{15}\alpha_{i}D_{i}$
are ample for some positive rational numbers $\alpha_{i}$.
Hence $\Phi$ preserves the ample cone.
\end{proof}

\begin{prop}
There exists an automorphism $\varphi$ on $X_{\text{AST}}$ which satisfies 
$\varphi (\Delta) = \Delta _{\text{AST}}$ and $\varphi \circ \sigma_{\text{AST}} =\sigma_{\text{AST}} \circ \varphi $.
\end{prop}
\begin{proof}
Since the fixed locus of $\sigma_{\text{AST}}$ is contained $\Delta$ by Lemma \ref{del-sta}
(see also the proof of Proposition \ref{le-k3uni}),
we may assume that $\Delta$ consists of components of singular fiber of type I$_{7}$, of type II$^{\ast}$
and a cross-section. Hence $\Delta$ is either 
\[ \Gamma_{2}-\Gamma_{3}-\Gamma_{4}-\Gamma_{5}-\Gamma_{6}-\Gamma_{7}-S
-\Theta_{1}-\Theta_{2}-\Theta_{3}-\Theta_{4}-\Theta_{5}-\Theta_{6}-\Theta_{7}-\Theta_{8}, \]
namely $\Delta _{\text{AST}}$, or
\[ \Gamma_{5}-\Gamma_{4}-\Gamma_{3}-\Gamma_{2}-\Gamma_{1}-\Gamma_{7}-S
-\Theta_{1}-\Theta_{2}-\Theta_{3}-\Theta_{4}-\Theta_{5}-\Theta_{6}-\Theta_{7}-\Theta_{8}.\]

By specifying (essentially relabeling) components of a singular fiber of type I$_{7}$, 
we can find an automorphism $\varphi$ on $X_{\text{AST}}$ satisfying $\varphi^{\ast}|S_{X_{\text{AST}}}=\Phi$ in Lemma \ref{isometry}.

We remark that $\sigma_{\text{AST}}$ acts trivially on $S_{X_{\text{AST}}}$ by Proposition \ref{id-7-S}.
Thus each action of  $\varphi \circ \sigma_{\text{AST}}$  and $\sigma_{\text{AST}} \circ \varphi$ on $H^{1,1}(X_{\text{AST}})$
is determined by $\varphi$ only.
Since $H^{2,0}(X_{\text{AST}})$ and $H^{0,2}(X_{\text{AST}})$ are both 1-dimensional,
$(\varphi \circ \sigma_{\text{AST}})^{\ast} =(\sigma_{\text{AST}} \circ \varphi)^{\ast} $ on $H^{2}(X, \mathbb{C})$.
Hence $\varphi \circ \sigma_{\text{AST}} =\sigma_{\text{AST}} \circ \varphi $ by the Torelli theorem.
\end{proof}

\end{document}